\documentclass[reqno,a4paper,11pt]{amsart}
\usepackage[utf8]{inputenc}

\usepackage[english]{babel}

\usepackage[T1]{fontenc}
\usepackage[normalem]{ulem}

\usepackage{amssymb}
\usepackage{amsmath,amsfonts}
\usepackage[tikz]{bclogo}
\usetikzlibrary{decorations.pathreplacing,calligraphy,quotes,calc,trees,positioning,arrows,arrows.meta,chains,shapes.geometric,
    decorations.pathreplacing,decorations.pathmorphing,shapes,
    matrix,shapes.symbols}
\usepackage[top=3cm,bottom=3cm,left=2.5cm,right=2.5cm]{geometry}

\usepackage{lmodern}
\usepackage{textcomp}
\usepackage{mathtools, bm}
\usepackage[unq]{unique}
\usepackage{comment}
\usepackage[breaklinks]{hyperref}
\usepackage{float}

\usepackage[numbers]{natbib}  

\usepackage{graphicx} % Required for inserting images

\usepackage{amsthm}

\usepackage{subcaption}

\usepackage[foot]{amsaddr}

\usepackage{cleveref}

\usepackage{bbm}

\usepackage{enumitem}
\usepackage{thmtools}
\usepackage{thm-restate}

\newtheorem{theorem}{Theorem}[section]
\newtheorem{lemma}[theorem]{Lemma}
\newtheorem{proposition}[theorem]{Proposition}

\newtheorem{Question}[theorem]{Question}

\newtheorem{claim}[theorem]{Claim}

\newtheorem{corollary}[theorem]{Corollary}

\theoremstyle{definition}

\newcommand{\one}{\mathbf{1}}
\DeclareMathOperator{\tr}{tr}
\DeclareMathOperator{\rank}{rank}

\newcommand{\calG}[2]{\mathcal{G}(#1, #2)}

\title{On sparse universally rigid graphs in $\mathbb{R}$}

\author{Ole Gabsdil$^{\star}$}
\email{ole.gabsdil@epfl.ch}

\author{Ezra Hirschi$^{\star}$}
\email{ezra.hirschi@epfl.ch}

\author{Julien Portier$^{\star}$}
\email{julienportier@laposte.net}

\address{$^{\star}$Department of Mathematics, EPFL, Lausanne, Switzerland}

\date{October 2026}

\begin{document}

\begin{abstract}
A graph is said to be 1-universally rigid if, for any generic realization of its vertices on the real line, the realization cannot be deformed out of the line in any higher-dimensional Euclidean space while preserving all edge lengths.
In this paper, we investigate the sparsity of 1-universally rigid graphs in different manners.

First, we show that an Erd\H{o}s-R\'enyi random graph $G(n,p)$ is 1-universally rigid with high probability whenever $p \geq C\log n \cdot \log\log n/n$ for some sufficiently large constant $C>0$.
These techniques also imply the existence of 1-universally rigid graphs of arbitrarily large girth, which resolves a problem suggested by Dantas e Moura, Jord\'an, and Silverman in a non-constructive manner.
We also study how sparse a 1-universally rigid graph can be, proving that there exists an absolute constant $\varepsilon>0$ such that every sufficiently large 1-universally rigid graph on $n$ vertices has at least $(3/2+\varepsilon)n$ edges.
This breaks the natural $3/2$ barrier arising from degree-2 extensions and makes progress towards a question of Jord\'an and Nguyen.
\end{abstract}

\maketitle

\section{Introduction}
A \textit{$d$-dimensional framework} $(G,\Phi)$\footnote{The more common notation throughout the literature is $(G,p)$, however as we will work with random graphs later on, we reserve the variable $p$ for probabilities.} is a pair consisting of a graph $G = (V,E)$ as well as a mapping $\Phi \colon V \rightarrow \mathbb{R}^d$ of its vertices into the euclidian space $\mathbb{R}^d$.
Two frameworks $(G,\Phi)$ and $(G,\Psi)$ are called \textit{equivalent} if $\lVert \Phi(u) - \Phi(v)\rVert_2 = \lVert \Psi(u) - \Psi(v)\rVert_2$ holds for all pairs of adjacent vertices $u$ and $v$.
They are called \textit{congruent} if $\lVert \Phi(u) - \Phi(v)\rVert_2 = \lVert \Psi(u) - \Psi(v)\rVert_2$ holds for all pairs of vertices $u$ and $v$, no matter whether they are connected via an edge or not.
Equivalent frameworks need not be congruent, since the distances between nonadjacent vertices may differ. This raises a natural question: when do the edge lengths of a framework uniquely determine the positions of all its vertices, up to an isometry of the ambient space?

Formally, we say that a $d$-dimensional framework $(G,\Phi)$ is \textit{globally rigid} in $\mathbb{R}^d$, if every equivalent $d$-dimensional framework $(G,\Psi)$ is congruent to $(G,\Phi)$.
We further say that a graph $G$ is \textit{$d$-globally rigid}, if $(G,\Phi)$ is globally rigid for all generic $d$-dimensional mappings $\Phi$, where generic means that the set of all coordinates $\Phi(V)$ is algebraically independent over the rationals.
This has received considerable attention in recent years, see for instance \cite{Connelly2009questions,connelly2011conjectures,Connelly2013,garamvolgyi2022global,garamvolgyi2023minimally,gortler2010characterizing,Hitting_time_higher_dimension}, to name a few.

A stronger notion of rigidity, which will be the focus of this work, is \textit{universal rigidity}. 
A $d$-dimensional framework $(G,\Phi)$ is called \textit{universally rigid}, if for any $d' \geq d$ the following holds: Interpreted as a $d'$-dimensional framework, every equivalent $d'$-dimensional framework is already congruent to $(G,\Phi)$.
In other words, universal rigidity requires a $d$-dimensional framework to be globally rigid in all dimensions at least $d$.
Similar to global rigidity, a graph is called \textit{$d$-universally rigid}, if all its generic $d$-dimensional embeddings are universally rigid.
Universal rigidity has been studied extensively, see for instance \cite{connelly2017universal,Connelly2020,gortler2014characterizing,oba2023characterizing}.
In this work, we focus on 1-universally rigid graphs. 
Related work on universal rigidity in one dimension can be found in \cite{chen_universal_2021,chen_universal_2022,Generic_Universal_Rigidity_in_R}.
The overarching question in this study is to determine what kind of (sparse) graphs contain sufficient amounts of certain structures making them 1-universally rigid.
We address this question considering multiple notions of sparsity, namely sparse random graphs, graphs with high girth and graphs with very few edges.\\

First, a natural question of interest is for which regimes of $p$ an Erd\H{o}s-R\'enyi random graph $G \sim \calG{n}{p}$ has a certain rigidity property.
This was first studied by Jackson, Servatius and Servatius \cite{First_Rigidity_of_random}, where among other results, the authors investigated the case of 2-global rigidity.
Recent results on global rigidity of random graphs in higher dimensions were obtained by \cite{Bootstrap_higher_dimension}, \cite{benjamini2024determiningpointsconfigurationline} and \cite{Hitting_time_higher_dimension}.
The one-dimensional case, which we will consider here, has also recently received interest.
Benjamini and Tzalik \cite{benjamini2024determiningpointsconfigurationline} proved that there exists some constant $C > 0$ such that for $p = C\log n/n$, a fixed 1-dimensional injective framework is globally rigid w.h.p.
This result was improved by Gir\~{a}o, Illingworth, Michel, Powierski and Scott \cite{Reconstrucing_in_R_sharp_constant}, who established a hitting-time result for reconstructing a fixed point configuration, implying in particular that $p=\log n/n$ is the sharp threshold.
Montgomery, Nenadov, Portier and Szab\'o \cite{Hitting_time_1-gloabl_rigidity_of_graph} proved the stronger result that, in fact, the hitting time at which every injective one-dimensional framework is globally rigid coincides w.h.p. with the hitting time for minimum degree two.
In this work we establish the following similar result concerning the 1-universal rigidity of random graphs:

\begin{theorem}\label{thm:Universal rigidity of random graph result}
There exists an absolute constant $C > 0$ such that $G \sim \calG{n}{p}$, for $p \geq C \log n \log \log n/n$, is 1-universally rigid w.h.p.
\end{theorem}

We note that Criscitiello \cite{criscitiello2025threshold} asked to determine the threshold for an Erd\H{o}s-R\'enyi random graph to be $d$-universally rigid. Thus, our result gives a partial answer to this question when $d=1$.
Moreover the proof we provide is robust to some amount of edge deletions, and therefore we also obtain the following result via a standard probabilistic argument.

\begin{theorem}\label{thm:high girth result}
There exist 1-universally rigid graphs of arbitrarily large girth.
\end{theorem}

Dantas e Moura, Jord\'an and Silverman \cite{Generic_Universal_Rigidity_in_R} recently suggested to find such 1-universally rigid graphs of arbitrarily large girth.
We achieve partial progress towards this goal by proving their existence in a non-constructive manner.

We remark that similar to the prior work mentioned above, within the proof of \Cref{thm:Universal rigidity of random graph result} we in fact prove the stronger statement that any \emph{injective} 1-dimensional framework $(G,\Phi)$ is universally rigid.
The genericity given by the common definition of universal rigidity stated above is not needed.
The same holds for \Cref{thm:high girth result}.\\

Complementing these results involving random graphs, we also study the 1-universal rigidity of very sparse deterministic graphs.
In particular, we achieve progress in determining the minimal number of edges needed to be 1-universally rigid.
The best currently known constructions have $|E| = 2|V|-3$ edges and can be obtained e.g. by iterative degree-2 extensions (see \Cref{lem:Deg-2-Extension}).
Despite the simplicity of this construction and the large number of examples it provides, it was speculated that $2|V|-3$ is actually the lowest possible number of edges:
\begin{Question} [\cite{tibor_ur_frameworks}, Question 3.6] Let $G=(V,E)$ be universally rigid in $\mathbb{R}^1$. Does this imply that $ |E|\geq 2|V|-3 $? \end{Question} 
There are several reasons to regard $2|V|-3$ as the natural candidate for the optimal lower bound. First, by Laman's theorem \cite{laman1970graphs}, every graph that is rigid in $\mathbb{R}^2$ has at least $2|V|-3$ edges, and  Jord\'an and Nguyen also speculated in \cite{tibor_ur_frameworks} that every graph that is universally rigid in $\mathbb{R}^1$ must be rigid in $\mathbb{R}^2$. %\julien{add definition of rigid}. I added it in the next sentence (Ole)
Here, unlike global rigidity, rigidity in $\mathbb{R}^2$ only requires that the respective frameworks do not admit nontrivial continuous deformations along equivalent frameworks.
Moreover, the bound $2|V|-3$ is known to be tight under the stronger assumption that every quasi-injective realization of $G$ on the line is universally rigid \cite{Generic_Universal_Rigidity_in_R}, where \emph{quasi-injective} means that adjacent vertices have distinct positions on the line.
Despite these indications, the best previously known lower bound for the general (generic) case is due to Dantas e Moura, Jord\'an, and Silverman \cite{Generic_Universal_Rigidity_in_R}, who proved that every 1-universally rigid graph on at least six vertices satisfies $|E|\geq \frac{3}{2}|V|$. 
Their proof essentially uses degree-2 extensions (which we will introduce in \Cref{sec:Preliminaries}) to reduce any 1-universally rigid graph to one of minimum degree 3.
Our next main result breaks this degree 3 barrier.

\begin{theorem}
\label{thm:Main-Deterministic}
    There exists an absolute constant $\varepsilon >0$ such that for all 1-universally rigid graphs on at least 7 vertices, we have
    \[|E| \geq (3/2+\varepsilon)|V|.\]
    In particular, one may take $\varepsilon=1/26$.
\end{theorem}

\subsection*{Paper organization.} In \Cref{sec:Preliminaries} we introduce some notation used throughout the paper, and state some rather simple and known results concerning the 1-universal rigidity of certain structures.
In \Cref{sec:Random part} we then focus on proving \Cref{thm:Universal rigidity of random graph result} and show that the proof can easily be adapted to also obtain \Cref{thm:high girth result}.
Following this, we provide a proof of \Cref{thm:Main-Deterministic} in \Cref{sec:Deterministic part}.
Lastly in \Cref{sec:Open problems}, we state some related open problems.

\subsection*{Acknowledgment.} This work was produced during the two semester projects of the first and second author as part of their M.Sc. and B.Sc. in Mathematics programs at EPFL, during spring semester 2025/26.
Both projects were supervised by the third author within the Chair of Extremal Combinatorics held by Prof. Oliver Janzer.
We would like to thank him for offering and organizing the projects in his group.

\subsection*{AI declaration.} All of the main ideas in this paper are due to the authors. The only use of AI concern \Cref{lem:random-properties} and \Cref{G1acyclic}. The authors first established \Cref{lem:random-properties} for a weaker bound on $p$ and then used ChatGPT 6 to optimise the associated computations, resulting in the bound stated in the present version.   
Concerning \Cref{G1acyclic}, the authors first produced a purely geometric proof, whose main idea was to start with an embedding of the whole graph in $\mathbb{R}$, and perturb one by one the vertices along the cycle into a new equivalent, non-congruent embedding in $\mathbb{R}^2$. 
Although this approach offers a natural geometric explanation for the lemma, closing the cycle while ensuring the initial embedding remains generic is technical, requiring a case analysis, and, in one case, a further perturbation into $\mathbb{R}^3$. The authors suspected there may be a simpler proof by using positive semidefinite equilibrium stresses, based on the characterization of universal rigidity by Gortler and Thurston \cite{gortler2014characterizing}. 
They asked ChatGPT 6 to provide such a proof, and the resulting proof is the one presented in this paper, with its exposition subsequently revised and improved by the authors.

\section{Preliminaries}\label{sec:Preliminaries}

\subsection*{Notation.}
For a graph \(G = (V,E)\), we write \(V(G)\) and $E(G)$ for its vertices and edges, as well as \(G[X]\) for the graph induced by a set $X$ of vertices.
By \(N(S)\) we denote the neighbourhood of a subset \(S \subset V(G)\), explicitly defined as \(N(S) = \left(\bigcup_{v \in S} N(v) \right) \setminus S\)  where \(N(v)\) are the vertices adjacent to \(v\). We use \(d(v)\) for the degree of \(v\), \(d_A(v)\) for the number of neighbours of \(v\) in \(A\), and \(e(A,B)\) for the number of edges between \(A\) and \(B\).

For the context of 1-universal rigidity, we say that two vertices \(v,w \in V(G)\) are \emph{universally \ linked} in a framework \((G, \Phi)\), if all equivalent frameworks \((G, \Psi)\) satisfy \(\lVert \Phi(v) - \Phi(w) \rVert = \lVert \Psi(v) - \Psi(w) \rVert\). Note that all the following lemmas assuming an edge between \(v,w\) also apply for \(v,w\) universally linked. Furthermore, a framework is 1-universally rigid if and only if all its vertices are universally linked.

\subsection*{Some results on 1-universal rigidity.}
The following relatively simple constructions and results concerning 1-universal rigidity will be used repeatedly throughout the paper.
We recall the following standard results, where we omit citations when the results are folklore.
\begin{lemma}\label{lem:K4 completion}
Let $G = K_4 \setminus \{uv\}$, that is the graph on 4 vertices missing exactly one edge $\{u,v\}$.
Then in any 1-dimensional injective framework $(G,\Phi)$, the vertices $u$ and $v$ are universally linked.
\end{lemma}

Note that these $K_4$-completions can also be applied to subframeworks, which is a useful tool in establishing 1-universal rigidity.
When used this way, a convenient way to think about them is in terms of degree-2 extensions.
We say that $G$ is a \emph{degree-2 extension} of $H$, if $G$ has exactly one additional vertex compared to $H$, which is adjacent to exactly two distinct vertices of \(H\).

\begin{lemma}
\label{lem:Deg-2-Extension}
Let $G$ and $H$ be two graphs, such that $G$ is obtained from $H$ by a degree-2 extension. Then $G$ is $1$-generically universally rigid if and only if $H$ is $1$-generically universally rigid.
\end{lemma}

The 'if'-direction of this statement follows immediately from $K_4$-completions.
The reverse direction is harder to prove and can be found in Theorem 8 of \cite{Generic_Universal_Rigidity_in_R}.\\
Next, we introduce a 1-universally rigid framework that we will use in this paper.

\begin{proposition}
Let $G = v_1 v_2\dots v_n v_1$ be a cycle.
A 1-dimensional framework $(G, \Phi)$ satisfying $\Phi(v_1) < \Phi(v_2) < \dots < \Phi(v_n)$ is 1-universally rigid.
\end{proposition}

We will refer to such frameworks as \emph{stretched cycles}.
Lastly, we also need the following classical Chernoff bound (see for instance Theorem 2.1 in \cite{janson2000random}).

\begin{lemma}
\label{lem:Chernoff}
Let \(X \sim \operatorname{Bin}(n,p)\) and let \(\lambda = np\). Then, for \(t>0\),
\[
\begin{aligned}
\mathbb{P}(X > \lambda + t)
&\leq
\exp\left(
-\frac{t^2}{2(\lambda+t/3)}
\right),
\qquad
\mathbb{P}(X < \lambda - t)
&\leq
\exp\left(
-\frac{t^2}{2\lambda}
\right).
\end{aligned}
\]
\end{lemma}

\section{Universal rigidity of random graphs in $\mathbb{R}$}\label{sec:Random part}

In this section we prove \Cref{thm:Universal rigidity of random graph result} and \Cref{thm:high girth result}, whose proofs are split up into two lemmas.
We start by establishing a sufficient condition for 1-universal rigidity in \Cref{lem:random graph rigid}.
Following that, we prove \Cref{lem:random-properties} stating that the random graph $G \sim G(n,p)$ satisfies a slightly stronger version of this condition w.h.p.
This immediately shows \Cref{thm:Universal rigidity of random graph result}. 
For \Cref{thm:high girth result} however, we need to delete edges from all short cycles before applying \Cref{lem:random graph rigid}.
Here, a suitable bound on the number of edge deletions will allow us to show that the condition of \Cref{lem:random-properties} still implies the weaker one in \Cref{lem:random graph rigid}.

\begin{lemma}\label{lem:random graph rigid}
\label{lem:deterministic}
Let $G$ be an $n$-vertex graph.  Suppose that there are integers $s,K,m$  and real $q$, such that
$Km\le n/6$, $s \leq m \leq n/15$, and $1 \leq s \leq q2^{K-1}$. 
Suppose that $G$ has the following
properties.
\begin{enumerate}
\item[(P1)] For every set $A$ with $m/2\le |A|\le m$, fewer than $m/4$ vertices of $V(G)\setminus A$ have
fewer than $q$ neighbours in $A$.
\item[(P2)] Every non-empty set $A$ with $|A|\le 2s$ spans fewer than $q|A|/3$ edges.
\item[(P3)] Every two disjoint sets of size $s$ have at least one edge between them.
\item[(P4)] Every nonempty set $A$ with $|A|\le 5n/6$ contains a vertex with at least two neighbours in
$V(G)\setminus A$.
\end{enumerate}
Then every injective one-dimensional framework of $G$ is universally rigid.
\end{lemma}

We start by explaining the proof idea for \Cref{lem:random graph rigid}.
The main strategy is to employ a bootstrap approach when proving universal rigidity of a framework, which is quite common in the field of rigidity of random graphs (see for instance \cite{Hitting_time_higher_dimension}).
More precisely, we will start by finding a universally rigid subframework of linear size, and then extend it to the remaining vertices one by one using $K_4$-completions.
The initial subframework will consist of conveniently overlapping stretched cycles.
To construct it we use some ideas from Bollob\'as, Brightwell, and Ne\v{s}et\v{r}il \cite{bollobas_many_deletions_for_cover}, who investigated the number of required edge deletions to obtain a graph admitting an embedding without any stretched cycle.
Their approach is to lower bound the number of stretched cycles (of certain length) in any embedding, upper bound the number of such cycles a fixed edge can lie in and divide these two quantities for the final bound.
While we can reuse some of the ideas within the stretched cycle construction leading to the lower bound, reducing to a count no longer works in our setting.
Instead, we need more insight on how the stretched cycles overlap, to derive universal rigidity of a linear sized component.

\begin{proof}[Proof of \Cref{lem:random graph rigid}]
Fix an injective mapping $\Phi:V(G)\to\mathbb R$, and split the vertices, according to their order under
$\Phi$, into six consecutive intervals $I_1,\dots,I_6$ of size $n/6$, ignoring harmless rounding errors.
Partition $I_3$ into $K$ consecutive blocks $J_1,\dots,J_K$ of size $m$.
This is possible as $Km \leq n/6$, and unused vertices are discarded.
The goal is to construct stretched cycles of the form displayed in \Cref{fig:Sketch strechted cycles main proof}, where the vertices in $I_3$ along the green path belong to consecutive blocks $J_k$.

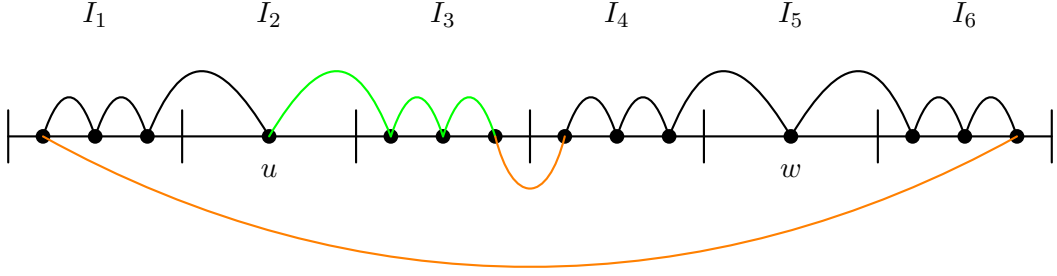
\begin{figure}
\begin{center}
\begin{tikzpicture}[scale=2.3,line cap=round,line join=round]

% Main horizontal line
\draw[thick] (-3,0) -- (3,0);

% End ticks
\draw[thick] (-3,0.15) -- (-3,-0.15);
\draw[thick] (-2,0.15) -- (-2,-0.15);
\draw[thick] (-1,0.15) -- (-1,-0.15);
\draw[thick] (0,0.15) -- (0,-0.15);
\draw[thick] (1,0.15) -- (1,-0.15);
\draw[thick] (2,0.15) -- (2,-0.15);
\draw[thick] (3,0.15) -- (3,-0.15);

% Points in line
\foreach \x in {-2.8,-2.5,-2.2,-1.5,-0.8,-0.5,-0.2,2.8,2.5,2.2,1.5,0.8,0.5,0.2} {
    \fill (\x,0) circle (1.2pt);
}

% Curves

% The small ones inside I_j
\draw[thick]
(-2.8,0)
.. controls (-2.7,0.3) and (-2.6,0.3) ..
(-2.5,0);

\draw[thick]
(-2.5,0)
.. controls (-2.4,0.3) and (-2.3,0.3) ..
(-2.2,0);

\draw[thick, color=green]
(-0.8,0)
.. controls (-0.7,0.3) and (-0.6,0.3) ..
(-0.5,0);

\draw[thick, color=green]
(-0.5,0)
.. controls (-0.4,0.3) and (-0.3,0.3) ..
(-0.2,0);

\draw[thick]
(2.8,0)
.. controls (2.7,0.3) and (2.6,0.3) ..
(2.5,0);

\draw[thick]
(2.5,0)
.. controls (2.4,0.3) and (2.3,0.3) ..
(2.2,0);

\draw[thick]
(0.8,0)
.. controls (0.7,0.3) and (0.6,0.3) ..
(0.5,0);

\draw[thick]
(0.5,0)
.. controls (0.4,0.3) and (0.3,0.3) ..
(0.2,0);

% The medium ones to I2 & I5

\draw[thick]
(-2.2,0)
.. controls (-2,0.5) and (-1.8,0.5) ..
(-1.5,0);

\draw[thick, color=green]
(-0.8,0)
.. controls (-1,0.5) and (-1.2,0.5) ..
(-1.5,0);

\draw[thick]
(2.2,0)
.. controls (2,0.5) and (1.8,0.5) ..
(1.5,0);

\draw[thick]
(0.8,0)
.. controls (1,0.5) and (1.2,0.5) ..
(1.5,0);

% The close-stretched-cycle ones

\draw[thick, color=orange]
(-0.2,0)
.. controls (-0.1,-0.4) and (0.1,-0.4) ..
(0.2,0);

\draw[thick, color=orange]
(-2.8,0)
.. controls (-1,-1) and (1,-1) ..
(2.8,0);

% Labels
\node at (-2.5,0.7) {$I_1$};
\node at (-1.5,0.7) {$I_2$};
\node at (-0.5,0.7) {$I_3$};
\node at (0.5,0.7) {$I_4$};
\node at (1.5,0.7) {$I_5$};
\node at (2.5,0.7) {$I_6$};

\node at (-1.5,-0.2) {$u$};
\node at (1.5,-0.2) {$w$};

\end{tikzpicture}
\caption{Sketch of stretched cycles we are looking for}
\label{fig:Sketch strechted cycles main proof}
\end{center}
\end{figure}

We now construct subsets $J_k'\subseteq J_k$ backwards, starting with $J_K'=J_K$.  
Having constructed $J_{k+1}'$ with
$|J_{k+1}'|\ge 3m/4$, property (P1) shows that all but fewer than $m/4$ vertices of $J_k$ have at least $q$
neighbours in $J_{k+1}'$.  Let $J_k'$ be these good vertices.  Thus $|J_k'|\ge 3m/4$, and every vertex of
$J_k'$ has at least $q$ neighbours in $J_{k+1}'$.  
Another application of (P1) gives a set
$U_+\subseteq I_2$, of size at least $n/6-m/4$, each of whose vertices has at least $q$ neighbours in $J_1'$.

Fix $u\in U_+$.  Let $N_k\subseteq J_k'$ be the set of vertices reachable from $u$ by a monotone path using one
vertex from each of $J_1',\dots,J_k'$.  We claim that
\begin{equation}\label{eq:growth}
 |N_{k+1}|\ge \min\{2|N_k|,s\}.
\end{equation}
Suppose for the sake of contradiction that \eqref{eq:growth} failed and choose $X\subseteq N_k$ with
\[
 |X|=\min\{|N_k|,\lceil s/2\rceil\}.
\]
Then $|X\cup N_{k+1}|\le 2s$.
We also clearly have $|N_{k+1}| \leq 2|X|$.
Hence (P2) gives
\[
 e\bigl(G[X\cup N_{k+1}]\bigr)<\frac{q}{3}(|X|+|N_{k+1}|)\leq3\frac{q}{3}|X| = q|X|.
\]
On the other hand, since every vertex of $X$ has
at least $q$ neighbours in $J_{k+1}'$ which by definition all lie in $N_{k+1}$, we have
\[
 e\bigl(G[X\cup N_{k+1}]\bigr)\ge q|X|,
\]
a contradiction.  Thus \eqref{eq:growth} holds.  Since $|N_1|\ge q$ and $q2^{K-1}\ge s$, every
$u\in U_+$ reaches at least $s$ vertices of $J_K \subseteq I_3$ by monotone paths.
Let $I_3^{u}$ be the set of these vertices.\\

The black parts of the stretched cycle shown in \Cref{fig:Sketch strechted cycles main proof} are constructed in the exact same way as we just constructed the green part.
We first repeat the construction from $I_2$ in the reverse direction through $I_1$.
This yields $U_-\subseteq I_2$ with
$|U_-|\ge n/6-m/4$, such that every $u \in U_-$  reaches a set of vertices $I_1^{u} \subseteq I_1$ of size at least $s$ by monotone decreasing paths. 
Setting $U := U_+ \cap U_-$, we get
\[
 |U|\ge n/6-m/2\ge 2s,
\]
where the last inequality follows from $s,m \leq n/15$.
The analogous construction around $I_5$, using $I_4$ and $I_6$, gives a set $W\subseteq I_5$ also of size at least $n/6-m/2\ge 2s$, with the following property: For any $w \in W$, the sets of vertices $I_4^{w} \subseteq I_4$ and $I_6^{w} \subseteq I_6$ reachable from $w$ via monotone paths, have size at least $s$.\\

Now fix $u\in U$ and $w\in W$. 
We are going to find the missing orange edges in \Cref{fig:Sketch strechted cycles main proof} to obtain a stretched cycle containing $u$ and $w$.
Since $I_3^{u}$ and $I_4^{w}$ both have size at least $s$, (P3) ensures that there is an edge between $I_3^{u}$ and $I_4^{w}$. 
Similarly, there is also an edge between $I_1^{u}$ and $I_6^{w}$.
By using the monotone paths from $u$ and $w$ to the endpoints of these edges, we obtain a stretched cycle containing $u$ and $w$.  Thus every pair in $U\times W$ is universally linked.\\

Since $|U|,|W|\ge 2s$, property (P3), applied to two disjoint $s$-subsets of $U$ and then of $W$, gives an
edge inside each of $U$ and $W$.  The $K_4$-completion lemma (\Cref{lem:K4 completion}) now shows that every pair in $W$ (using the universally linked pair in $U$) and $U$ is universally linked. Consequently, $U\cup W$ is a
universally rigid component of size at least $2(n/6 - m/2) \geq n/6$.\\

It remains to absorb all remaining vertices.  Let $C$ be a maximal universally linked clique containing
$U\cup W$.  Assume for the sake of contradiction that $C\ne V(G)$, and let $A=V(G)\setminus C$.  Since $|C|\ge |U\cup W| \geq n/6$, we have
$1\le |A|\le 5n/6$.  By (P4), some $v\in A$ has two neighbours $x,y\in C$.  
Hence, by \Cref{lem:K4 completion}, $C\cup\{v\}$ is a universally linked clique.
This contradicts the maximality of $C$, thus we have $C=V(G)$ and the framework is universally rigid.
\end{proof}

We will now show that suitable Erdős–Rényi random graphs satisfy the assumptions of \Cref{lem:deterministic}.
We start by setting up the parameters.
Choose constants satisfying the following:
\begin{equation*}
    C \gg H \gg S \gg 1.
\end{equation*}
With these constants we define the following parameters:
\begin{align*}
    p := C\frac{\log n \cdot \log \log n}{n}, \qquad
    K:=\lceil 4\log n\rceil,\qquad
    m:=\left\lfloor\frac{n}{6K}\right\rfloor,\\
    s:=\left\lceil\frac{Sn}{C\log n}\right\rceil, \qquad
    q:=\frac{np}{100K}, \qquad 
    h:=H  \log \log n
\end{align*}

\noindent For these parameters, we obtain the following:
\begin{lemma}\label{lem:random-properties}
The Erdős–Rényi random graph $G \sim G(n,p)$ satisfies the following properties w.h.p.
\begin{enumerate}
\item[(Q1)] For every set $A$ with $m/2\le |A|\le m$, fewer than $m/4$ vertices of $V(G)\setminus A$ have
fewer than $2q$ neighbours in $A$.
\item[(Q2)] Every non-empty set $A$ with $|A|\le 2s$ spans fewer than $h|A|$ edges.
\item[(Q3)] Every two disjoint sets of size $s$ have at least $S\log \log n \cdot s/2$ edges between them.
\item[(Q4)] Every nonempty set $A$ with $|A|\le 5n/6$ contains a vertex with at least three neighbours in
$V(G)\setminus A$.
\end{enumerate}
\end{lemma}

\begin{proof}
We verify the four properties separately. \\

\paragraph{Property (Q1).}
Let $\mathcal{E}_1$ be the event that (Q1) fails.
If $\mathcal{E}_1$ holds, then there are disjoint sets $A,T$ with $m/2\le |A|\le m$ and $|T|= \lceil m/4 \rceil$ such that every vertex of $T$ has fewer than $2q$ neighbours in $A$.  For fixed $A,T$, each such degree follows a
binomial distribution with mean at least $p m/2$, and $2q\le pm/8$ for large $n$.  Chernoff's inequality (see \Cref{lem:Chernoff}) and independence
of the edge sets incident from distinct vertices of $T$ into $A$ give
\[
 \mathbb{P}\bigl(d_A(v)<2q\text{ for all }v\in T\bigr)
 \le \exp(-\Omega(pm|T|))\le \exp(-\Omega(pm^2)).
\]
Using $\binom{N}{t} \leq (eN/t)^t$, the number of possible pairs $(A,T)$ is at most
\[
 \sum_{i=m/2}^m \binom{n}{i} \binom{n}{m/4} = \exp\bigl(O(m\log(n/m))\bigr)=\exp(O(m \log \log n)).
\]
Thus by union bound, and using $pm^2=\Theta(Cm  \log \log n)$, we obtain
\begin{align*}
    \mathbb{P}(\mathcal{E}_1) \leq \exp(O(m \log \log n)-\Omega(pm^2)) = \exp(O(m \log \log n)-\Omega(Cm \log \log n)).
\end{align*}
For $C$ sufficiently large the RHS tends to zero. This proves the desired bound. \\

\paragraph{Property (Q2).}
Let $\mathcal{E}_2$ be the event that (Q2) fails.
Fix $a\le 2s$.  If $a<2h$, then $ha>\binom a2$, so no $a$-set can span $ha$ edges.  Assume $a\ge 2h$.
For a fixed $a$-set $A$, we have
\[
 \mathbb{P}(e(A)\ge ha) 
 \le \binom{\binom{a}{2}}{ha}p^{ha} \leq \left(\frac{eap}{2h}\right)^{ha}.
\]
Thus by union bound we have
\[
 \mathbb{P}(\mathcal{E}_2) \leq \sum_{2h \leq a \leq 2s}\binom{n}{a}\left(\frac{eap}{2h}\right)^{ha} \leq \sum_{2h \leq a \leq 2s}\left[\frac{en}{a}\left(\frac{eap}{2h}\right)^h\right]^a = \sum_{2h \leq a \leq 2s} \exp(a f(x_a)),
\]
where $x_a=ap/\log\log n$ and $f(x)=\log\frac{eC\log n}{x}+H \log \log n\log\frac{ex}{2H}$.  
For large enough $n$, note that $f$ is increasing, and thus $f(x_a) \leq f(2S+o(1))$.  
Since $H \gg S $, we have $f(2S+o(1)) \leq -cH \log \log n$, for some constant $c>0$.
Thus 
\[
 \mathbb{P}(\mathcal{E}_2) \leq \sum_{2h \leq a \leq 2s} \exp(-cH a \log \log n ) = o(1).
\]

\paragraph{Property (Q3).}
Let $\mathcal{E}_3$ be the event that (Q3) fails.
For fixed disjoint $s$-sets $A,B$, the number of edges between them follows a binomial distribution of mean
\[
 \mu=ps^2 \geq S s \log \log n.
\]
Chernoff's inequality gives
\[
 \mathbb{P}(e(A,B)<\mu/2)\le e^{-\mu/8}.
\]
Thus
\begin{align*}
    \mathbb{P}(\mathcal{E}_3) \leq \binom ns^2e^{-\mu/8} \leq \exp\bigl(O(s \log \log n)-\Omega(Ss \log \log n) \bigr),
\end{align*}
 where we have used the bound $\binom{n}{s} \leq (\frac{en}{s})^s$.
Taking $S$ sufficiently large, the RHS is $o(1)$, as wanted.\\

\paragraph{Property (Q4).}
Let $\mathcal{E}_4$ be the event that (Q4) fails.
Fix $B$ with $1\le b:=|B|\le 5n/6$.  
For every $v\in B$, the number of neighbours of $v$ outside $B$ follows a binomial distribution of mean at least $np/6$.
The relevant edge sets are independent as $v$ varies over $B$, and by Chernoff inequality (see \Cref{lem:Chernoff})
\[
 \mathbb{P}\bigl(d_{V\setminus B}(v)\le 2\bigr)\le \exp(-np/20),
\]
for large $n$.  Hence the probability that every vertex of $B$ has at most two neighbours outside $B$ is at
most $\exp(-npb/20)$.  Summing over all sets $B$ of the relevant size gives
\[
\mathbb{P}(\mathcal{E}_4) \leq \sum_{b=1}^{5n/6}\binom nb e^{-npb/20}
 \le \sum_{b\ge1}\exp\bigl(b\log(en/b)-npb/20\bigr)=o(1),
\]
because $np=C\log n\log\log n\gg\log n$. This finishes the proof.
\end{proof}

\begin{proof}[Proof of \Cref{thm:Universal rigidity of random graph result}]
Using the parameters defined above, \Cref{lem:random-properties} shows that $G \sim G(n,p)$ satisfies (Q1)-(Q4) w.h.p.
Recalling that $h < q/3$, we note that these properties trivially imply (P1)-(P4) of \Cref{lem:random graph rigid}.
The necessary inequalities in \Cref{lem:random graph rigid} follow directly from the definitions of the parameters, when $n$ is chosen large enough.
Hence \Cref{lem:random graph rigid} guarantees that $G\sim G(n,p)$ is 1-universally rigid w.h.p.
\end{proof}

We are also able to prove \Cref{thm:high girth result} in a similar way, where we only need to also account for some edge deletions.

\begin{proof}[Proof of \Cref{thm:high girth result}]
Fix an integer $g$ for which we will prove the existence of a 1-universally rigid graph of girth at least $g$.
We use the parameters from above, where we may assume that $n$ is large enough w.r.t. all of them including $g$.\\
For our probabilistic proof we define a graph distribution $H_g(n,p)$, where $H \sim H_g(n,p)$ is obtained by first sampling $G \sim G(n,p)$ and then randomly removing one edge from each cycle of length shorter than $g$.
We will prove that such an $H \sim H_g(n,p)$, which obviously has girth at least $g$, is still 1-universally rigid w.h.p.
Let $G \sim G(n,p)$ be the graph $H$ was obtained from.
By \Cref{lem:random-properties}, we may assume that $G$ satisfies the properties (Q1)-(Q4).
To establish \Cref{thm:high girth result}, we will show that in this case $H$ satisfies (P1)-(P4) of \Cref{lem:random graph rigid} w.h.p.\\

First note that (P2) trivially follows from (Q2), as edge deletions can only lower edge counts.
To establish the other properties, we claim that each vertex in $G$ lies in at most one cycle shorter than $g$ w.h.p.
This means that w.h.p. at most one incident edge of each vertex is deleted, hence (P1), (P3) and (P4) follow from (Q1), (Q3) and (Q4) respectively, by noting that $q \gg 1$ and $S\log \log n \cdot s/2 \gg s$.\\
To prove the claim, let $\mathcal{E}$ be the event that there exists a vertex that lies on two distinct cycles of length smaller than $g$.
Taking the union of these two cycles, we obtain a graph consisting of $a \leq 2g-3$ vertices and at least $a+1$ edges.
For a fixed set of $a$ vertices, this happens with probability at most $c(g)p^{a+1}$, where $c(g)$ is a constant depending only on $g$.
Taking a union bound over all possible sets of size at most $2g-3$ yields
\begin{equation*}
    \mathbb{P}[\mathcal{E}] \leq \sum_{a=1}^{2g-3} \binom{n}{a}c(g)p^{a+1} \leq c(g)p \sum_{a=1}^{2g-3} (np)^{a} \leq c(g) \frac{(C\log n \log \log n)^{2g-1}}{n} = o(1),
\end{equation*}
establishing the claim.
\end{proof}

\section{Minimal density of universally rigid graphs in $\mathbb{R}$}\label{sec:Deterministic part}

This section is devoted to proving \Cref{thm:Main-Deterministic}.
We will use the following edge-replacement operation.
Let $H_1$ be a graph, let $uv\in E(H_1)$, and
let $H_2$ be a graph with two distinct distinguished
vertices $a,b$, called its terminals.
We obtain $G$ by deleting the edge $uv$ from $H_1$
and adding a copy of $H_2$, identifying $a$ with
$u$ and $b$ with $v$.
We call this an \emph{edge-replacement}.
This edge-replacement operation has some interesting properties with respect to 1-universal rigidity, as shown in the following lemma.

\begin{lemma}\label{lem:edge-replacement}
Let $G$ be obtained from $H_1$ by replacing an
edge $uv$ with a graph $H_2$ with terminals $a,b$.
Then the following hold:
\begin{enumerate}
    \item If $H_1$ is not 1-universally rigid, then $G$ is not 1-universally rigid.
    \item If $H_1$ and $H_2$ are both
    1-universally rigid, then so is $G$.
\end{enumerate}
\end{lemma}

\begin{proof}
For the first statement, since $H_1$ is
not 1-universally rigid, then there exists a generic line realization
$\Phi$ of $H_1$ and an equivalent, non-congruent
realization $\Psi$ in some Euclidean space.
Extend $\Phi$ to a generic line realization of
$G$ by embedding the internal vertices of $H_2$ arbitrarily
at generic positions.
Since $uv\in E(H_1)$, we have
\[
    |\Phi(u)-\Phi(v)|
    =\|\Psi(u)-\Psi(v)\|.
\]
Hence there exists an isometry from the line
of $\Phi$ to the line through $\Psi(u)$ and
$\Psi(v)$ that maps $\Phi(u)$ to $\Psi(u)$
and $\Phi(v)$ to $\Psi(v)$.
Using this isometry to place the internal
vertices of $H_2$, we extend $\Psi$ to an
equivalent realization of $G$.
The two realizations are not congruent,
since their restrictions to $H_1$ are not.
Thus $G$ is not 1-universally rigid.

For the second statement, suppose that both
$H_1$ and $H_2$ are 1-universally rigid.
Let $\Phi$ be a generic line realization of
$G$ and let $\Psi$ be any equivalent realization.
Since $H_2$ is 1-universally rigid, its
restriction is congruent to its realization
in $\Phi$. In particular,
\[
    \|\Psi(u)-\Psi(v)\|
    =|\Phi(u)-\Phi(v)|.
\]
Thus the restriction of $\Psi$ to $H_1$,
with the edge $uv$ restored, is equivalent
to the restriction of $\Phi$ to $H_1$.
Since $H_1$ is 1-universally rigid, these
restrictions are congruent.
Moreover, the realizations of $H_1$ and
$H_2$ in $\Psi$ are both congruent to their
line realizations in $\Phi$ and share the two
distinct vertices $u,v$.
Therefore, they lie on the same line and
together form a realization congruent to
$\Phi$. 
Hence $G$ is 1-universally rigid.
\end{proof}

\noindent The following figure shows an example application of \Cref{lem:edge-replacement}.

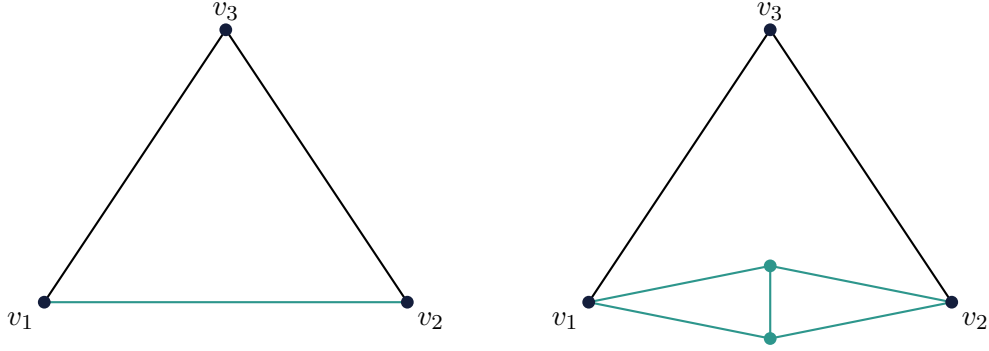
\begin{figure}[H]
\centering
\label{fig:(K_3,K_4-e)}
\begin{tikzpicture}[scale=1.2]

% Colours
\definecolor{mygreen}{RGB}{45,150,140}
\definecolor{myorange}{RGB}{245,120,60}
\definecolor{myblue}{RGB}{55,70,155}
\definecolor{myblack}{RGB}{20,30,60}

%%%%%%%%%%%%%%%%%%%%%%%%%%%%%%%%%%%%%%%%%%%%%%%%%%
%% TOP LEFT TRIANGLE
%%%%%%%%%%%%%%%%%%%%%%%%%%%%%%%%%%%%%%%%%%%%%%%%%%

\coordinate (A1) at (2,0);
\coordinate (A2) at (6,0);
\coordinate (A3) at (4,3);

\draw[thick] (A1)--(A3);
\draw[thick] (A3)--(A2);
\draw[mygreen, thick] (A1)--(A2);

\fill[myblack] (A1) circle (2pt);
\fill[myblack] (A2) circle (2pt);
\fill[myblack] (A3) circle (2pt);

\node[below left] at (A1) {$v_1$};
\node[below right] at (A2) {$v_2$};
\node[above] at (A3) {$v_3$};

%%%%%%%%%%%%%%%%%%%%%%%%%%%%%%%%%%%%%%%%%%%%%%%%%%
%% TOP RIGHT GRAPH
%%%%%%%%%%%%%%%%%%%%%%%%%%%%%%%%%%%%%%%%%%%%%%%%%%

\begin{scope}[xshift=8cm]

% Outer vertices
\coordinate (B1) at (0,0);
\coordinate (B2) at (4,0);
\coordinate (B3) at (2,3);

% Green internal
\coordinate (C1) at (2,0.4);
\coordinate (C2) at (2,-0.4);

% Black edges
\draw[thick] (B1)--(B3);
\draw[thick] (B3)--(B2);

% Green edges
\draw[mygreen, thick] (B1)--(C1);
\draw[mygreen, thick] (C1)--(B2);
\draw[mygreen, thick] (C1)--(C2);
\draw[mygreen, thick] (C2)--(B1);
\draw[mygreen, thick] (C2)--(B2);

% Vertices
\fill[myblack] (B1) circle (2pt);
\fill[myblack] (B2) circle (2pt);
\fill[myblack] (B3) circle (2pt);

\fill[mygreen] (C1) circle (2pt);
\fill[mygreen] (C2) circle (2pt);

\node[below left] at (B1) {$v_1$};
\node[below right] at (B2) {$v_2$};
\node[above] at (B3) {$v_3$};

\end{scope}

\end{tikzpicture}
\caption{\Cref{lem:edge-replacement} applied to $H_1 = K_3$ (shown on the left with its replaced edge colored green) and $H_2 = K_4-e$. As $H_1$ and $H_2$ are both 1-universally rigid, so is the graph obtained on the right.}
\end{figure}

\noindent We now show the following result.

\begin{lemma}\label{deg3triangle}
    Let \(G=(V,E)\) be a graph of minimum degree three, with $G \neq K_4$. 
    Suppose \(G\) contains three vertices $v_1$, $v_2$ and $v_3$ which form a triangle in $G$, each of degree three, and whose neighbours outside the triangle are also of degree three. Then \(G\) is not 1-universally rigid.
\end{lemma}
\begin{proof}
    Let \(C = \{v_1,v_2,v_3\}\) be the given triangle of degree 3 vertices. For \(i=1,2,3\), let \(u_i\) be the neighbour of \(v_i\) not in \(C\). By symmetry, it suffices to deal with the three following cases. Note that only case 2 requires the assumption that the \(u_i\) have degree three.

\textbf{Case 1: \(u_1=u_2=u_3\).} 
In this case $v_1,v_2,v_3$ and $u_1$ form a $K_4$.
As $G \neq K_4$ and $v_1,v_2,v_3$ and $u_1$ are all of degree 3, they form an connected component in $G$.
Hence $G$ is not connected and cannot be 1-universally rigid.

\textbf{Case 2: \(u_1=u_2 \neq u_3\).} Let $x$ be the neighbour of \(u_1\) that is not in $C$. If $x=u_3$, then $G$ is not 2-connected and therefore not 1-universally rigid. Thus $x \neq u_3 $ and \(G\) can be obtained from the 4-cycle \(u_1v_3u_3xu_1\) by replacing the edge \(u_1v_3\) with \(G[C\cup\{u_1\}]\) and the edge \(u_3x\) with \(G[V\setminus(C\cup\{u_1\})]\). Since the 4-cycle is not 1-universally rigid, by \Cref{lem:edge-replacement}, we have that \(G\) is not 1-universally rigid.

\textbf{Case 3: \(u_1, u_2, u_3\)} are all distinct. Let $B_3$ be the graph consisting of two copies of \(K_3\) connected by 3 disjoint edges. As explained in detail in \cite{chen_universal_2021} and \cite{chen_universal_2022}, the graph $B_3$ is not 1-universally rigid, and thus there are two equivalent, non-congruent frameworks \((B_3,\Phi)\) and \((B_3,\Psi)\), where \(\Phi\) is generic and embeds $B_3$ on the real line. We can therefore also find two equivalent, non-congruent frameworks \((G,\Phi')\), \((G,\Psi')\) where \(\Phi'\) is generic by taking the triangle \(C\) to be one copy of \(K_3\) in \(B_3\) and the neighbours \(u_1, u_2, u_3\) to be the other copy of \(K_3\) and placing all the other vertices at generic points on the line in \(\Phi'\), and at the same position on the line passing through \(u_1, u_2, u_3\) in $\Psi'$.
\end{proof} 

\noindent From the proof of the previous lemma, one can easily deduce the following result.

\begin{corollary}
\label{cor:deg3triangle_case_3}
Let \(G\) be a graph that contains a triangle whose three vertices each have degree three. If their unique neighbours outside the triangle are pairwise distinct, then $G$ is not 1-universally rigid.
\end{corollary}
\begin{proof}
    This is case 3 of the proof of Lemma \ref{deg3triangle} without unnecessary assumptions.
\end{proof}

\noindent We now prove the following technical lemma, which is at the heart of our proof of \Cref{thm:Main-Deterministic}.

\begin{lemma}
\label{G1acyclic}
    Let $G$ be a graph containing an induced cycle $C=v_1 \dots v_k v_1$ such that $k \geq 4$. Suppose that every vertex $v_i$ has exactly one neighbour $u_i$ outside of $C$, and that $u_i \neq u_{i+1}$ for every $i \in [k]$ (where the indices are considered mod $k$). Then $G$ is not 1-universally rigid.
\end{lemma}

\begin{proof}
Suppose, for a contradiction, that $G$ is $1$-universally rigid, and let $n=|V(G)|$.
By the theorem of Gortler and Thurston~\cite{gortler2014characterizing},
for every generic line placement $p$ of $G$ there is a symmetric
positive-semidefinite equilibrium stress matrix $\Omega$ satisfying
\[
    \Omega\one=0,\qquad \Omega p=0,\qquad
    \rank\Omega=n-2,
\]
with $\Omega_{vw}=0$ whenever $v\ne w$ and $vw\notin E(G)$.
Equilibrium at $v$ gives
\begin{equation}
\label{eq:vertex-equilibrium}
    \sum_{w:\,vw\in E(G)}
    \Omega_{vw}\bigl(p(v)-p(w)\bigr)=0.
\end{equation}

We start with a compactness argument to allow a
non-generic initial embedding.
More precisely, we construct an embedding $p^0$ as follows.
For each distinct vertex $u$ among
$u_1,\ldots,u_k$, choose a real number $y(u)$ such that $y(u)\neq y(w)$ whenever $u\neq w$.
If all $u_i$ are distinct, we take
\[
    y(u_k)=0,\qquad y(u_1)=1,\qquad
    y(u_2)=3,\qquad y(u_3)=2.
\]
For every $i$, set
\[
    x_i:=y(u_i),\qquad
    p^0(v_i)=x_i,\qquad
    p^0(u_i)=x_i-1.
\]
We place the remaining vertices arbitrarily.

Choose generic placements $p^{(r)}\to p^0$, and let $\Omega^{(r)}$
be the corresponding positive-semidefinite stress matrices of rank
$n-2$. 
Let $M^{(r)}=\Omega^{(r)}[C,C]$, that is the submatrix of $\Omega^{(r)}$ corresponding to the vertices of $C$. Note that $M^{(r)}$ is positive semidefinite, and also nonzero, as otherwise this would contradict $\rank \Omega^{(r)}=n-2$.

Rescale the stresses so that $\tr M^{(r)}=1$.
Since the set of positive-semidefinite $k\times k$ matrices of trace one
is compact, by passing to a subsequence, we may assume that
\begin{align*}
%\label{eq:limit}
    M^{(r)}\longrightarrow M,\qquad
    M\succeq0,\qquad \tr M=1.
\end{align*}
Set $a_j=-M_{j,j+1}$ and $b_j=(M\one)_j$, with indices taken modulo $k$.
Since $C$ is induced, we have
\begin{align}
\label{eq:matrix-form}
    M_{jj}=a_{j-1}+a_j+b_j,\qquad
    M_{j,j+1}=M_{j+1,j}=-a_j,
\end{align}
with all other off-diagonal entries zero. 
Since $u_j$ is the unique neighbour of $v_j$ outside $C$,
the zero row sums of $\Omega^{(r)}$ give
\[
    \Omega^{(r)}_{v_j u_j}
    =-(M^{(r)}\one)_j\longrightarrow-b_j.
\]
Thus all entries involved in equilibrium at $v_j$ converge.
Taking the limit in the equation $(\Omega^{(r)} p^{(r)})_j=0 $ and using $p^0(v_j)=x_j$ and $p^0(u_j)=x_j-1$,
we obtain for every $j$ that
\begin{equation}
\label{eq:limit-equilibrium}
    a_{j-1}(x_j-x_{j-1})
    +a_j(x_j-x_{j+1})+b_j=0.
\end{equation}
We now show that all $b_j$ vanish.
We sum \eqref{eq:limit-equilibrium} over all $j$.
Each cycle-edge contribution occurs twice with opposite signs, so
all these contributions cancel, leaving
\[
    \sum_{j=1}^k b_j=0.
\]
Since $M\one=(b_1,\dots,b_k)^T$, we deduce
\[
    \one^TM\one=\sum_{j=1}^k b_j=0.
\]
As $M\succeq0$, we have $M\one=0$, and thus $b_j=0$ for every $j$. \\ 

We now finish the proof by showing that all
$a_j$ vanish.
First, note that our choice of coordinates
ensures that $x_j\neq x_{j+1}$ for every $j$.
We now claim that the cyclic sequence
$x_1,\ldots,x_k$ contains at least two strict
increases and at least two strict decreases. 
Indeed, if all the $u_i$ are distinct, then our explicit choice of $x_k$, $x_1$, $x_2$ and $x_3$ ensures this is the case.
Otherwise, $u_i=u_j$ for
some distinct $i,j$. 
Then $x_i=x_j$, and both paths along the cycle joining
$v_i$ and $v_j$ start and end at the same
coordinate.
Since consecutive coordinates are distinct,
each of these paths must contain at least
one strict increase and one strict decrease.

Since $b_j=0$ for every $j$, by \eqref{eq:limit-equilibrium} we have
\[
    a_{j-1}(x_j-x_{j-1})
    =a_j(x_{j+1}-x_j),
\]
and thus there exists a constant
$t\in\mathbb{R}$ such that
\[
    a_j(x_{j+1}-x_j)=t
\]
for every $j$.
If $t=0$, then $a_j=0$ for every $j$,
since $x_j\neq x_{j+1}$.
Together with $b_j=0$, this implies
$M=0$, contradicting $\tr M=1$.
Thus we may assume that $t\neq 0$.
Since the cyclic sequence $(x_j)$ contains
at least two strict increases and at least
two strict decreases, there are two
distinct indices $r,s$ such that both $a_r$ and $a_s$ are negative.
We now show that this contradicts that $M$ is positive semidefinite.

Recall from \eqref{eq:matrix-form} that,
since $b_j=0$ for every $j$,
\[
    M_{jj}=a_{j-1}+a_j,\qquad
    M_{j,j+1}=M_{j+1,j}=-a_j,
\]
with all other entries zero.
Therefore, for any vector
$z=(z_1,\ldots,z_k)^T\in\mathbb{R}^k$:
\begin{align*}
    z^TMz
    =\sum_{j=1}^k
    (a_{j-1}+a_j)z_j^2
    -2\sum_{j=1}^k a_jz_jz_{j+1}
    =\sum_{j=1}^k a_j
    \left(z_j^2+z_{j+1}^2
    -2z_jz_{j+1}\right)=\sum_{j=1}^k
    a_j(z_j-z_{j+1})^2.
\end{align*}

Now removing the two edges $v_rv_{r+1}$ and $v_sv_{s+1}$ splits $C$ into
two paths.
Let $z\in\{0,1\}^k$ be the indicator
vector of the vertices of one of these paths.
By construction, the endpoints of every
edge of $C$ other than $v_rv_{r+1}$
and $v_sv_{s+1}$ have the same value
under $z$.
Hence
\[
    (z_j-z_{j+1})^2
    =
    \begin{cases}
        1, & j\in\{r,s\},\\
        0, & \text{otherwise}.
    \end{cases}
\]
It follows that
\[
    z^TMz
    =\sum_{j=1}^k
    a_j(z_j-z_{j+1})^2
    =a_r+a_s<0,
\]
which contradicts that $M$ is positive semidefinite, and finishes the proof.
\end{proof}

\noindent We are now ready to prove \Cref{thm:Main-Deterministic}.

\begin{proof}[Proof of \Cref{thm:Main-Deterministic}]
    Let $G=(V,E)$ be a 1-universally rigid graph on at least 7 vertices.
    Let $H$ be a graph obtained by successive deletions of degree 2 vertices in $G$, such that the graph $H$ contains no vertex of degree 2.
    By \Cref{lem:Deg-2-Extension}, $H$ is still 1-universally rigid, which results in 3 cases: Either $H = K_2$, or $H$ has minimum degree 3 and at most 6 vertices, or $H$ has minimum degree 3 and at least 7 vertices.
    In the first two cases we have $|E(H)| \geq 2|V(H)|-3$ and since $G$ is obtained from $H$ by degree-2 extensions we also have $|E| \geq 2|V|-3 \geq (3/2 + 1/26)|V|$ for \(|V| \geq 7\).
    In the last case we may continue the proof assuming that all vertices in $G$ already had degree at least 3.

    We set \(S:= \{v \in V | d(v)\geq 4\}\), \(S' := N(S) \cup S\) and \(V_1 := V\backslash S'\). 
    Let \(m:=\sum_{v \in V} (d(v) -3) = 2|E| - 3|V|\) be the number of 'extra degrees' that \(G\) has compared to a 3-regular graph on the same number of vertices. 
    Note that $m = \sum_{v \in S} (d(v) -3) \geq \sum_{v \in S} 1 =|S|$, and 
    \(|N(S)| \leq \sum_{v \in S} d(v)  = m+3|S| \leq 4m\).  
    Therefore 
    \begin{align}
    \label{eq:Ineq-V1}
        |V_1| \geq |V| - |S| -|N(S)| \geq |V|-5m.
    \end{align}
    Let $G_1$ be the graph induced by $G$ on $V_1$. 
    We now make the following claim.

    \begin{claim}
    \label{cl:G1acyclic}
        $G_1$ is acyclic.
    \end{claim}
    
\begin{proof}
    Suppose for the sake of contradiction that $G_1$ contains a cycle, and let $C$ be a shortest such cycle. Let $C=v_1v_2\dots v_kv_1$, and let $u_i$ be the unique neighbour of $v_i$ outside of $C$. By \Cref{deg3triangle}, we must have $|C| \geq 4$, as otherwise $G$ would not be 1-universally rigid.\\
Now we further assume that $u_i = u_{i+1}$ for some $i$.
This vertex has exactly 3 neighbours, namely $v_i, v_{i+1}$ and a third vertex $x$.
If $x = v_{i-1}$ (or symmetrically $x = v_{i+2}$), the triangle $v_iu_iv_{i+1}$ satisfies the requirements of \Cref{deg3triangle} and hence $G$ is not 1-universally rigid.
In the remaining case $x \neq v_{i-1}, x \neq v_{i+2}$ we note that also $v_{i-1} \neq v_{i+2}$ as $|C| \geq 4$, hence $G$ is not 1-universally rigid due to \Cref{cor:deg3triangle_case_3}.
We conclude that $u_i \neq u_{i+1}$ for all $i$.\\
Hence $C$ satisfies the requirements of \Cref{G1acyclic} and $G$ is not 1-universally rigid. 
This is a contradiction, so $G_1$ has to be acyclic. 
\end{proof}

Since \(G_1\) is acyclic by \Cref{cl:G1acyclic}, it follows that \(G_1\) has at most \(|V_1|\) edges. 
Since by definition every vertex in \(V_1\) has degree 3 in $G$, it follows that 
\begin{align}
\label{eq:UB-crossedges}
    e_G(S',V_1) = (\sum_{v \in V_1} d(v)) - 2e(G_1) \geq 3|V_1|-2|V_1| \geq  |V|-5m ,
\end{align}
where the last inequality follows by \eqref{eq:Ineq-V1}.
%and \(G_1\) has at most \(|V_1|-1\) edges so \(G_1\) has at least \(|V|-5m+2\) edges to \(S'\). 
Since \(N(S) \cap V_1 = \emptyset\) by definition, we have that $e(S',V_1)=e(N(S),V_1)$. 
Each vertex in \(N(S)\) has degree 3 and at least one neighbour in \(S'\), so that $|N(v) \cap V_1| \leq 2$ for every $v \in N(S)$. Therefore
\begin{align}
\label{eq:LB-crossedges}
    e(S',V_1)=e(N(S),V_1) \leq 2|N(S)| \leq 8m.
\end{align}
Combining \eqref{eq:UB-crossedges} and \eqref{eq:LB-crossedges}, we obtain
\[8m \geq |V|-5m.\]
Plugging in \(m = 2|E|-3|V|\) and rearranging, we obtain
\[|E| \geq (20/13)|V| = (3/2 + \varepsilon)|V|,\]
for \(\varepsilon = 1/26\), as wanted.
\end{proof} 

\section{Open problems}\label{sec:Open problems}

\noindent We finish by stating some open problems.

\subsection*{Sharp threshold for 1-universal rigidity of random graphs.}
In view of \Cref{thm:Universal rigidity of random graph result}, a natural question to ask is whether the bound obtained there on $p$ is tight.
To the best of our knowledge, the only known lower bound on the threshold for $G \sim \calG{n}{p}$ to be 1-universally rigid, is the known threshold for 1-global rigidity, which coincides with the hitting time for minimum degree $2$ by \cite{Hitting_time_higher_dimension}.
This leaves a gap of a factor of $\Theta(\log\log n)$ between the best known lower bound and our upper bound. 
We do not see clear evidence for either of the bounds to be tight. 
Note that the problem remains open for the stronger property that every injective 1-dimensional framework is 1-universally rigid.
As mentioned previously our upper bound \Cref{thm:Universal rigidity of random graph result} still holds for this stronger notion.
However, to our knowledge, the best known lower bound due to \cite{Hitting_time_1-gloabl_rigidity_of_graph} also coincides with the hitting time for minimum degree $2$, leaving the same gap as in the generic case.

\subsection*{Sharp threshold for $d$-universal rigidity of random graphs.}
First, we notice that there is no natural extension of \Cref{thm:Universal rigidity of random graph result} to higher dimensions without imposing a genericity assumption. 
Indeed, one may consider an embedding where $n-2$ points are on a line, while two vertices, say $u$ and $v$, lie outside of this line, so that the distance $uv$ cannot be determined, unless the edge $uv$ belongs to the graph. 
Thus any graph with the property that all of its injective $d$-dimensional frameworks are universally rigid with $d \geq 2$, must be complete.
However, the question posed by Criscitiello \cite{criscitiello2025threshold} remains open in the restricted generic setting.
Let $p_c(n,d)$ be the threshold for $d$-generic universal rigidity.
 The techniques used to prove \Cref{thm:Universal rigidity of random graph result} do not seem to extend directly to higher dimensions, since there is no known sparse higher-dimensional universally rigid analogue of stretched cycles.
 To the best of our knowledge, the sparsest known example of such a graph that is universally rigid in \(\mathbb{R}^d\) is \(K_{d+3}\) with one edge removed. 
 Combining this observation with results on bootstrap percolation in random graphs (see \cite{balogh2012graph}, \cite{bartha2024weakly} and \cite{bartha2025sharp}), we find that $p_c(n,d) =O_d( n^{-1/\eta(d)})$, where $\eta(d):=\frac{d+4}2-\frac1{d+1}$.
For a lower bound, we have $p_c(n,d) = \Omega_d(\log n/n)$, since a $d$-generic universally rigid graph with at least $d+2$ vertices has minimum degree at least $d+1$.
 To the best of our knowledge, those are currently the best known bounds on $p_c(n,d)$. 

\subsection*{1-Universally rigid graphs of arbitrarily large girth.}
As discussed above, \Cref{thm:high girth result} gives a non-constructive proof of the existence of 1-universally rigid graphs of arbitrarily high girth, and thus provides only limited insights concerning the structural properties of such graphs. We believe that it would be interesting to find explicit examples of 1-universally rigid graphs of arbitrarily high girth, which likely was closer in spirit to the initial suggestion by Dantas e Moura, Jord\'an, and Silverman \cite{Generic_Universal_Rigidity_in_R}. 

\subsection*{Minimal number of edges of 1-universally rigid graphs.}
While our bound in \Cref{thm:Main-Deterministic} breaks the natural $3/2$-barrier in this problem arising from degree-2 extensions, our bound remains far from the best-known upper bound of $2|V|-3$.
Nevertheless, our methods show that any lower bound of the form $|E|\geq (2-o(1))|V|$ for 1-universally rigid graphs would imply the tight bound $|E|\geq 2|V|-3$. 
Indeed, suppose there exists a 1-universally rigid graph $G_0$ such that $|E(G_0)|<2|V(G_0)|-3$. Then we can recursively construct a sequence of graphs $G_0,G_1,\ldots$, where $G_i$ is obtained from $G_{i-1}$ by replacing one edge with a copy of $G_0$. 
By \Cref{lem:edge-replacement}, every $G_i$ is 1-universally rigid.
By construction we have $|E(G_i)|=|E(G_0)|+i(|E(G_0)|-1)$ as well as  $|V(G_i)|=|V(G_0)|+i(|V(G_0)|-2)$.
This yields $|E(G_i)|/|V(G_i)| \rightarrow (|E(G_0)|-1)/(|V(G_0)|-2) < 2$, where the assumption $|E(G_0)|<2|V(G_0)|-3$ is used in the very last step.
This would contradict any lower bound of the form $|E|\geq (2-o(1))|V|$.

\bibliographystyle{abbrv}
\bibliography{bib}

\end{document}